\documentclass[12pt]{amsart}
\usepackage[utf8]{inputenc}
\usepackage{graphicx}

\usepackage{mathtools, bm}
\usepackage{amssymb, bm}
\usepackage{amsthm}
\usepackage{amsfonts}
\usepackage{stmaryrd}
\usepackage{hyperref, enumerate}
\usepackage{amsmath}
\usepackage[T1]{fontenc}
\usepackage{setspace}
\usepackage{dsfont}
\usepackage{array, color}
\usepackage{fancybox}

\newcommand{\Z}{\mathbb{Z}}

\newcommand{\Q}{\mathbb{Q}}

\newcommand{\Sgothique}{\mathfrak{S}}

\newcommand{\Agothique}{\mathfrak{A}}

\newtheorem{thm}{Theorem}[section]
\newtheorem{Lemma}[thm]{Lemma}

\newtheorem{rmq}[thm]{Remark}
\newtheorem{prop}[thm]{Proposition}

\newtheorem{Conjecture}[thm]{Conjecture}

\begin{document}
\title[]{On the height of some generators of Galois extensions with big Galois group}
\author{Jonathan Jenvrin}
\address{Jonathan Jenvrin, Univ. Grenoble Alpes, CNRS, IF, 38000 Grenoble, France. \textit{E-mail adress :} \href{mailto:jonathan.jenvrin@univ-grenoble-alpes.fr}
{\texttt{jonathan.jenvrin@univ-grenoble-alpes.fr}}
}
\date{}
\begin{abstract} We study the height of generators of Galois extensions of the rationals having the alternating group $\mathfrak{A}_n$ as Galois group. We prove that if such generators are obtained from certain, albeit classical, constructions, their height tends to infinity as $n$ increases. This provides an analogue of a result by Amoroso, originally established for the symmetric group.
\end{abstract}

\maketitle

\section{Introduction}
In this article, we let $\overline{\mathbb{Q}}$ be a fixed algebraic closure of $\mathbb{Q}$.
Given $\alpha\in \overline{\mathbb{Q}}$ of degree $d$, we denote by $M(\alpha)$ its Mahler measure, defined as
$$M(\alpha) = |a| \prod_{i=1}^{d} \max\left(1, |\alpha_{i}|\right) \geq 1$$
where $a$ is the leading coefficient of the minimal polynomial of $\alpha$ over $\mathbb{Z}$, and $\alpha_{1}, \dots, \alpha_{d}$ are the conjugates of $\alpha$.
The logarithmic Weil height of $\alpha$, or, for short, the height of $\alpha$ is then given by
$$h(\alpha) = \frac{\log(M(\alpha))}{d}.$$

While, by Kronecker's theorem, it is well-known that $h(\alpha) = 0$ if and only if $\alpha = 0$ or $\alpha$ is a root of unity,  Lehmer's conjecture predicts the existence of a positive constant $c > 0$ such that
$$ h(\alpha) \ge \frac{c}{d}$$
whenever $h(\alpha)$ is not zero.
The conjecture has been proved for various classes of algebraic numbers, but it is still open in general. 
The most notable progress toward Lehmer's conjecture is Dobrowolski's result \cite[Theorem 1]{Dobrowolski}, which proves that, if $h(\alpha)\neq 0$, then
\[ h(\alpha) \ge \frac{c}{d}\left( \frac{\log\log d}{\log d} \right)^{3}\]
(one can take $c=1/4$ as shown in \cite[Theorem on p. 83]{voutier}). 

If, on the one hand, Dobrowolski's theorem stands as the sole unconditional result on this problem, on the other hand, it is possible to show that specific classes of algebraic numbers satisfy even stronger variants of Lehmer's conjecture, such as the Bogomolov property introduced by Bombieri and Zannier in  \cite{bombieriZannier2001note}:  a set of algebraic numbers $S$ satisfies the \emph{Bogomolov property (B)} if there exists a constant $c=c(S)>0$ such that for every $\alpha\in S$ either $h(\alpha)=0$ or $h(\alpha)\geq c$.

Property (B) holds, for instance, for abelian extensions of number fields, as shown in \cite{Amoroso_Zannier-RelativeDobrowolskiLowerBoundOverAbelianExtensions,Amo_Zannier-DihedralCase} and for the field of totally real algebraic numbers, as proved in  \cite[Corollary 1]{Schinzel-TotalementReel}.

A set of algebraic numbers which has attracted the attention in this respect in recent years, is the following 
\[S_{\mathrm{Gal}}=\{\alpha\in \overline{\Q}\mid \Q(\alpha)/\Q\text{ is Galois }\}.\]

While Amoroso and David  proved that Lehmer's conjecture holds for the elements of this set (see \cite[Corollary 1.7]{Amoroso_David-LehmerPbEnDimensionSuperieur}), later Amoroso and Masser \cite[Theorem 3.3]{Amo_MasserheightInGaloisExtensions} showed that $S_{\mathrm{Gal}}$ satisfies even the following stronger result: 
for any $\epsilon > 0$, there exists a positive effective constant $c(\epsilon)$ such that, for every $\alpha \in S_{\mathrm{Gal}}$ of degree $d$ over $\Q$ and not a root of unity,  one has
$$h(\alpha) \ge c(\epsilon)d^{-\epsilon}.$$

A result this strong might prompt the question, of whether the set $S_\mathrm{Gal}$ satisfies Property (B) or not.

A quite natural way to tackle this issue is by fixing the Galois group of $\Q(\alpha)/\Q$. 
As recalled before, the answer is known to be positive when such group is abelian (see also \cite{Amoroso_Robert-LowerBoundForTheHeightInAbelianExtensions}), while the dihedral case has been considered in \cite[Corollary 1.3]{Amo_Zannier-DihedralCase}.

Motivated by a question posed by Smyth (see \cite[Problem 21]{BIRSworkshop}), Amoroso studied certain classes of generators of Galois extensions whose Galois group is  the full symmetric group $\mathfrak{S}_{n} $ and proved the following: 
\begin{thm}[{\cite[Theorems 1.1 and 1.2]{amoroso2018mahler}}]\label{thm main}
Let $\beta$ be an algebraic integer of degree $n\geq 3$, let $\beta_{1}, \ldots, \beta_{n}$ be its conjugates and assume that $\mathrm{Gal}(\mathbb{Q}(\beta_{1}, \ldots, \beta_{n})/\mathbb{Q}) = \mathfrak{S}_{n}$. 

Let $a_{1},\ldots,a_{n}$ be integers. Then: 
\begin{enumerate}[(i)]
    \item\label{(i)} If $\beta$ is a unit, $\alpha=\beta_{1}^{a_{1}}\ldots \beta_{n}^{a_{n}}$ is a generator of $\Q(\beta_{1},\ldots,\beta_{n})/\Q$ if and only if $a_{1},\ldots,a_{n}$ are pairwise distinct. In this case $h(\alpha)$ tends to infinity as $n$ increases.
    \item $\alpha'=a_{1}\beta_{1}+\ldots+a_{n}\beta_{n}$ is a generator of $\Q(\beta_{1},\ldots,\beta_{n})/\Q$ if and only if $a_{1},\ldots,a_{n}$ are pairwise distinct. In this case $h(\alpha')$ tends to infinity as $n$ increases.
\end{enumerate}
\end{thm}

This result led Amoroso to propose the following conjecture:
\begin{Conjecture}[{ \cite[Conjecture 1.3]{amoroso2018mahler}}]\label{conj-Amo}
Let $\alpha \in \overline{\mathbb{Q}}$ be a generator of a Galois extension of degree $d = n!$ of Galois group $\mathfrak{S}_{n}$. Then $h(\alpha) \ge c(d)$
with $c(d)$ a function tending to infinity with $d$.
\end{Conjecture}

The goal of this article is to investigate whether some analogue of Theorem \ref{thm main} holds for other groups. Our first result provides a lower bound for the height of generators of Galois extensions as in \eqref{(i)} with no assumption on the structure of the Galois group of the extension.

\begin{thm} \label{generalresult}
    Let $n\geq 3$ be an integer and let $\beta$ be an algebraic integer of degree $n$, with conjugates $\beta_{1}, \ldots, \beta_{n}$. Let $a_{1}, \ldots, a_{n} \in \mathbb{Z}$, and let
    $ \alpha= \beta_{1}^{a_{1}} \ldots \beta_{n}^{a_{n}}.$
  Suppose that $\Q(\alpha)=\Q(\beta_{1},\dots,\beta_{n})$. Then
    $$h(\alpha) \ge  \left({\frac{1}{n}  \left|\sum_{i=1}^{n} a_{i}\right|}\right) \cdot \log \left |N_{\Q(\beta)/\Q}(\beta) \right|.$$

In particular, if $\beta$ is not a unit and if  the sequence $\left(\frac{1}{n} \left|\sum_{i=1}^{n}a_{i} \right| \right)_n$ tends to infinity with $n$, then  $h(\alpha)$   tends also to infinity with $n$.
\end{thm}

The remainder and most significant part of this article is devoted to proving the analogue of Theorem \ref{thm main} for extensions having Galois group equal to the alternating group $\mathfrak{A}_{n}$. We remark that this case, besides being naturally interesting to examine, can be also consider in some sense the \emph{generic} one. For instance, it has been proved in \cite[Theorem 2]{PolyAvecCoefPetitGroupDeGaloisAn} that, for any integer $\ell\geq 2$, a random polynomial of degree $n$ with i.i.d. random integers coefficients taking values uniformly in $\{1,\ldots,\ell\}$ will have Galois group containing  $\mathfrak{A}_{n}$ with a probability tending to $1$ as $n$ goes to infinity.

Our second main result 
is the analogue of \cite[Theorem 1.1]{amoroso2018mahler} for $ \mathfrak{A}_{n}$, for generators obtained as products of conjugates.
\begin{thm} \label{thm 1.1}
Let $\beta$ be an algebraic number of degree $n \ge 5$, let $\beta_{1}, \ldots, \beta_{n}$ be its conjugates and assume that $\mathrm{Gal}(\mathbb{Q}(\beta_{1}, \ldots, \beta_{n})/\mathbb{Q}) = \mathfrak{A}_{n}$. Let $\alpha= \beta_{1}^{a_{1}} \ldots \beta_{n}^{a_{n}}$ where $a_{1}, \ldots, a_{n} \in \mathbb{Z}$. Then:

\begin{enumerate}[(1)]
\item\label{thm1.1-it1} $\alpha$ is a generator of $\mathbb{Q}(\beta_{1}, \ldots, \beta_{n})/\mathbb{Q}$ if and only if there are at most two distinct indices $i, j$ such that $a_{i} = a_{j}$.

\item\label{thm1.1-it3} If $\alpha$ is a generator of $\mathbb{Q}(\beta_{1}, \ldots, \beta_{n})/\mathbb{Q}$ and $\beta$ is a unit, then

$$h(\alpha) \ge (1+g(n)) \sqrt{\frac{n}{200\pi}}\left(\frac{\log(\log(n))}{\log(n)}\right)^{3}$$
where $g(n)$ tends to 0 as $n$ tends to infinity. 
In particular, $h(\alpha)$ tends to infinity with $n$.
\end{enumerate}
\end{thm}
Our last main result, a counterpart to \cite[Theorems 1.2]{amoroso2018mahler} for $\mathfrak{A}_{n}$, deals with generators obtained as linear combination of conjugates:
\begin{thm}\label{thm 1.2} Let $\beta$ be an algebraic integer of degree $n \ge 5$, let $\beta_{1}, \ldots, \beta_{n}$ be its conjugates and suppose that $\mathrm{Gal}(\mathbb{Q}(\beta_{1}, \ldots, \beta_{n})/ \mathbb{Q}) = \mathfrak{A}_{n}$.  Consider $\alpha=a_{1} \beta_{1}+\ldots+a_n \beta_n$ where $a_{1},\ldots,a_{n} \in \mathbb{Z}$. Then: 
\begin{enumerate}[(1)]
    \item\label{thm 1.2-it1} $\alpha$ is a generator of $\mathbb{Q}(\beta_{1}, \ldots, \beta_{n})/ \mathbb{Q}$ if and only if there are at most two distinct indices $i,j$ such that $a_{i}=a_{j}$.
    \item\label{thm 1.2-it2} If $\alpha$ is a generator of $\mathbb{Q}(\beta_{1}, \ldots, \beta_{n})/ \mathbb{Q}$, then 
    $$h(\alpha)\ge  {\frac{1}{240}}\log\left(\frac{n}{9} \right).$$ In particular, $h(\alpha)$ tends to infinity with $n$.
\end{enumerate}
\end{thm}
We notice that our results potentially support an extension of Conjecture \ref{conj-Amo} to generators of Galois extensions with Galois group containing $\mathfrak{A}_{n}$.

We now describe the structure of the article and briefly outline the main ideas involved in the proofs of our results.

The proof of Theorem \ref{generalresult}, presented in Section \ref{sect-general-res}, relies on elementary group theory and basic properties of the Mahler measure. It is notably simpler than the proof of Theorem \ref{thm 1.1}, which suggests that the most challenging scenario arises when $\beta$ is a unit.

The proofs of Theorems \ref{thm 1.1} and \ref{thm 1.2} are presented in Sections \ref{Générateur de An version multiplicative} and  \ref{Générateur de An version additive}, respectively. 

Our proof strategy for Theorem \ref{thm 1.1} builds upon Amoroso's approach in \cite[Theorem 1.1]{amoroso2018mahler}. 
Point \eqref{thm1.1-it1} is proved in Section  \ref{sec-proof-gen-mult} and  relies on an adaption of a result of Smyth (see Proposition \ref{prop pour montrer le point (1)}) and some elementary Galois theoretic arguments.

Point \eqref{thm1.1-it3} is a corollary of a more precise statement, given in Proposition \ref{thm1.1-it2}, whose proof, done in Section \ref{point2-thm1.1}, is more subtle and requires some technicalities.
First, in Proposition \ref{Mahler-sn}, we link the Mahler measures of $\alpha$ and $\beta$ to the values taken by the  function 
\[ \frac{2}{n!}\sum_{\sigma \in \Agothique_{n}} \left|\frac{1}{n}\sum_{j=1}^{n} \log|\beta_{\sigma{(j)}}| \left(a_{j} - \frac{1}{n}\sum_{i=1}^{n}a_{i}\right) \right|.\]

We then study the properties of this function and provide  an upper and lower bound for it (see Proposition \ref{(2.1)}) in terms of a constant $c_n$ (estimated in Proposition \ref{(2.2)}).
Point \eqref{thm1.1-it3} then follows combining the above results with the explicit version of Dobrowolski's theorem by Voutier \cite[Theorem, p. 83]{voutier}.

As for Theorem \ref{thm 1.2}, the proof of point \eqref{thm 1.2-it1}, carried out in Section \ref{point-1-theorem1.3},
 is the same, in additive notation, as the one of point \eqref{thm1.1-it1} of Theorem \ref{thm 1.1} and it is even simpler than its multiplicative variant. 

It is in the proof of point \eqref{thm 1.2-it2}, done in Section \ref{sec-point(2)-thm1.3}, that lies the main novelty of our strategy. 
The key point here is Lemma \ref{M(tau(alpha)) en fonction de M(alpha)}
 which describes how the Mahler measure of $\alpha$ changes when one \emph{applies} a transposition to $\alpha$, so, by hypothesis, an element not in $\mathrm{Gal}(\Q(\alpha)/\Q)$ and that \emph{should not act on $\alpha$}.  
This is then used to give a lower bound for $M(\alpha)$ (see Proposition \ref{minoration de M(alpha)})
 independent on the chosen transposition, but depending, amongst others, on the quantity
$$V(\mathfrak{a})=\prod_{1\le i<j\le n}(a_{j}-a_{i}).$$
This bound is exploitable only when all $a_i$'s are distinct, and in this case the quantities involved can be bounded using, among other ingredients, estimates on the derangement numbers.

However, $\alpha$ could be a generator of the extension even if two of the indices $a_{i}$'s are equal, causing the quantity $V(\mathfrak{a})$ to be zero and providing no useful information. To overcome this, we construct another generator given by a linear combination of the $\beta_i$'s with all distinct integral coefficients and having height close to that of $\alpha$. This is done using the key Lemma \ref{(a_i-a_j)(beta_i-beta_j)}. The height of such a generator can then be bounded applying the first part of the proof, and this allows to conclude.

Finally, in Section \ref{section-examples}, we present some applications of our results to specific families of polynomials. Additionally, we discuss whether the proof strategy of our main results can be adapted to generators of Galois extensions with Galois groups different from $\mathfrak{A}_n$ or $\mathfrak{S}_n$.

\section{Proof of Theorem \ref{generalresult}}\label{sect-general-res} 
For the proof of Theorem \ref{generalresult}, we will need the following elementary lemma.
\begin{Lemma} \label{lemme élémentaire transitif}
    Let $G < \mathfrak{S}_{n}$ be a transitive subgroup. For $ i,k \in \{1,\dots,n\}$ let $G_{i,k}=\{\sigma \in G\mid \sigma(k)=i \}$. 
Then $|G_{i,k}|=\frac{|G|}{n}.$
\end{Lemma}

\begin{proof}
    Let $i,j,k \in \{1,\dots,n \}$ be fixed, and let $\sigma \in G$.
    Since $G$ is transitive, we can take $\tau \in G$ such that $\tau(i)=j$.
    We define the function 
    $$\begin{array}{ccccc}
        \phi_{\tau} & : &  G_{i,k}  & \to &  G_{j,k}  \\
        & & \sigma & \mapsto & \tau \sigma \\
    \end{array}.$$
 One can check that $\phi_{\tau}$ is well-defined, and it is a bijection, with inverse given by $\sigma \mapsto \tau^{-1} \sigma$. So we have $|G_{i,k}|=|G_{j,k}|$. Since $$\sum_{\ell=1}^{n} |G_{\ell,k}|=|G|$$ we have the sought-for equality. 
\end{proof}

We can now prove Theorem \ref{generalresult}.
\begin{proof}[Proof of Theorem \ref{generalresult}]
     Let $G$ be the Galois group of $\Q(\beta_{1},\dots,\beta_{n})$ over $\Q$.
   Notice that, by the definition of the Mahler measure, we have
    \begin{equation}\label{bound-M-alpha}M(\alpha) \ge \left |N_{\Q(\alpha)/\Q}(\alpha) \right |.\end{equation}

    Our goal is to show that
    \begin{equation}\label{goal}\left |N_{\Q(\alpha)/\Q}(\alpha) \right |^{1/|G|}=\left |N_{\Q(\beta)/\Q}(\beta) \right |^{\frac{1}{n}\sum_{i=1}^{n}a_{i}}.\end{equation}

    We have
   \begin{equation}\label{bound-norm-1}
        \left |N_{\Q(\alpha)/\Q}(\alpha) \right | = \prod_{\sigma \in G} \prod_{i=1}^{n} \left | \beta_{i}^{a_{\sigma^{-1}(i)}} \right | = \prod_{i=1}^{n} \left | \beta_{i}^{\sum_{\sigma \in G} a_{\sigma^{-1}(i)}} \right |=\prod_{i=1}^{n} \left | \beta_{i}^{\sum_{\sigma \in G} a_{\sigma(i)}} \right |.
\end{equation}

    Since $G$ is a transitive subgroup of $\Sgothique_{n}$,  for all $i \in \{1,\dots,n\}$, we have
    \[\sum_{\sigma \in G}a_{\sigma(i)}=\sum_{\sigma \in G}a_{\sigma(1)}.\]

    Therefore
   \[
        \prod_{i=1}^{n} \left | \beta_{i}^{\sum_{\sigma \in G} a_{\sigma(i)}} \right | = \prod_{i=1}^{n} \left | \beta_{i} \right |^{\sum_{\sigma \in G} a_{\sigma(1)}}  
        = \left | N_{\Q(\beta)/\Q}(\beta) \right |^{\sum_{\sigma \in G} a_{\sigma(1)}}.\]

    Letting $G_{i,1}=\{\sigma \in G\mid \sigma(1)=i \}$, by Lemma \ref{lemme élémentaire transitif} we have 
    \[\sum_{\sigma \in G} a_{\sigma(1)}=\sum_{i=1}^n \left(\sum_{\sigma \in G_{i,1}} a_i\right)=\frac{|G|}{n}\sum_{i=1}^{n} a_{i}.\]

    Hence, from \eqref{bound-norm-1} we obtain
    \begin{equation*}\left |N_{\Q(\alpha)/\Q}(\alpha) \right | = \left | N_{\Q(\beta)/\Q}(\beta) \right |^{\frac{|G|}{n}\sum_{i=1}^{n} a_{i}}.\end{equation*}
    This proves \eqref{goal} by taking the $|G|$-th root.
Finally, from \eqref{bound-M-alpha}, we have
    \[M(\alpha)^{1/|G|} \ge \left |N_{\Q(\beta)/\Q}(\beta) \right| ^{\frac{1}{n} \sum_{i=1}^{n} a_{i}}.\]

    Applying this to $1/\alpha$ and using $M(\alpha)=M(1/\alpha)$, we also have
    \[M(\alpha)^{1/|G|} \ge \left |N_{\Q(\beta)/\Q}(\beta) \right| ^{-\frac{1}{n} \sum_{i=1}^{n} a_{i}}.\]

    This proves Theorem \ref{generalresult} by noticing that $|G|=[\Q(\alpha):\Q]$.

\end{proof}

\section{Proof of Theorem \ref{thm 1.1}} \label{Générateur de An version multiplicative}

The objective of this section is to establish Theorem \ref{thm 1.1}.
To achieve this, as mentioned earlier, we will follow the proof outlined by Amoroso in \cite[Sections 2 and 3]{amoroso2018mahler} and adapt it as needed.

For enhanced readability, the proof of the  distinct points in Theorem \ref{thm 1.1} is divided into  separate subsections. Within each subsection, we introduce the necessary technical results and definitions needed in the proof of the corresponding point.

Throughout the article, if $\beta_1,\ldots,\beta_n\in \overline{\Q}$ is a full set of Galois conjugates, when writing $\Agothique_{n} \subseteq \mathrm{Gal}(\Q(\beta_{1},\dots,\beta_{n})/\Q)$, we identify $\sigma\in \Agothique_{n}$ with the automorphism $\sigma(\beta_i)=\beta_{\sigma(i)}$.

\subsection{The multiplicative case: proof of Theorem \ref{thm 1.1}, point \eqref{thm1.1-it1} }\label{sec-proof-gen-mult}
We start, as in \cite{amoroso2018mahler}, by adapting the multiplicative version of \cite[Lemma 1]{smyth1986additive}:
\begin{prop} \label{prop pour montrer le point (1)}
Let $\beta$ be an algebraic number of degree $n\geq 5$ such that $\beta^{k} \notin \Q$ for any non-zero integer $k$.
Let $\beta_{1},\dots,\beta_{n}$ be the conjugates of $\beta$, and suppose that $\Agothique_{n} \subset \mathrm{Gal}(\Q(\beta_{1},\dots,\beta_{n})/\Q)$. Then, for every $v_{1},\dots,v_{n} \in \Z$, not all equal, the product $\beta_{1}^{v_{1}} \dots \beta_{n}^{v_{n}}$ {is not a root of unity}.
\end{prop}

\begin{proof}
Assume, by contradiction, that $\beta_{1}^{v_{1}} \dots \beta_{n}^{v_{n}}$ is a root of unity. Up to multiplying all the $v_i$'s by the same non-zero integer,  we can suppose that $\beta_{1}^{v_{1}} \dots \beta_{n}^{v_{n}}=1$. 

 Notice that by hypothesis, $\mathrm{Gal}(\Q(\beta_1,\ldots,\beta_n)/\Q)$ contains all 3-cycles. Without loss of generality, we can suppose  that $\beta=\beta_1$ and $v_{1} \neq v_{2}$.

Applying the cycle $(1 \; 3 \; 2)$ to both sides of the equality $\beta_{1}^{v_{1}} \dots \beta_{n}^{v_{n}}=1$, we obtain $\beta_1^{v_1-v_2}\beta_2^{v_2-v_3}\beta_3^{v_3-v_1}=1$.
Applying to this last equality the cycle $(1 \; i \; j)$, for $i\neq j$ and $i,j\geq 4$, we obtain $\beta_1^{v_1-v_2}=\beta_i^{v_1-v_2}$ for all $i\geq 4$. 

Applying further the cycles $(1 \; 2 \; k)$ and $(1 \; 3 \; \ell)$ for $k,\ell\not\in\{2,3,i\}$, we finally get that for all $1\leq i\leq n$, $\beta_{1}^{v_{1}-v_{2}}=\beta_{i}^{v_{1}-v_{2}}$. So, $\beta^{v_{1}-v_{2}} \in \Q$ as it is fixed by all elements of $\mathrm{Gal}(\Q(\beta_{1},\dots,\beta_{n})/\Q)$, contradicting  the hypothesis on $\beta$ as $v_1\neq v_2$.
\end{proof}

Now we are able to prove point $(1)$ of Theorem \ref{thm 1.1}.

\begin{proof}[Proof of Point $\eqref{thm1.1-it1}$ of Theorem \ref{thm 1.1}]
Assume first that $\alpha$ is a generator of $\Q(\beta_1,\ldots,\beta_n)/\Q$. 
If $|\{a_{i}\mid 1 \le i \le n \}| < n-1$, then either there exist indices $i<j<k$ such that $a_{i}=a_{j}=a_{k}$, or there are two couples $(i,j)\neq (k,\ell)$ with $i<j$ and $k<\ell$ such that $a_{i}=a_{j}$ and $a_{k}=a_{\ell}$.

In the first case, we have that $\alpha$ is fixed by the 3-cycle $(i \; j \; k)$, while in the second case $\alpha$ is fixed by the double transposition $(i \; j)(k \; \ell)$. 
In all cases, we find $[\Q(\alpha): \Q]<|\Agothique_{n}|$, contradicting the fact that  $\alpha$ is a generator of $\Q(\beta_1,\ldots,\beta_n)/\Q$.

For the converse, assuming that $|\{a_{i}, 1 \le i \le n \}| \geq n-1$, we want to show that the group $\mathrm{Gal}(\Q(\beta_{1},...,\beta_{n})/ \Q(\alpha))$ is trivial.

Let $\sigma\in \mathrm{Gal}(\Q(\beta_{1},...,\beta_{n})/ \Q(\alpha))$ and let $\tau=\sigma^{-1}$. We want to show that $\tau$ is the identity.
Regarding $\sigma$ as an element of $\Agothique_{n}$, the equality $\sigma(\alpha)=\alpha$ gives  \begin{equation}\label{eq-alpha-sigma(alpha)}\prod_{k=1}^{n}\beta_{k}^{a_{k}-a_{\tau(k)}}=1.\end{equation}

Notice now that $\beta$ satisfies the hypothesis of Proposition \ref{prop pour montrer le point (1)}. Indeed, suppose that $\beta^{k} \in \Q$ for some integer $k>0$. Then, for all $i,j$, we have $\beta_{i}^{k} = \beta_{j}^{k}$. So, for all $i,j$, there exists a $k$-th root of unity $\zeta_{i,j,k}$ such that $\beta_{i} = \zeta_{i,j,k}\beta_{j}$. We have for all $i,j$ 
\begin{equation} \label{argument galois correspondence}
    \zeta_{i,j,k} = \frac{\beta_{i}}{\beta_{j}} \in \mathbb{Q}(\beta_{1},\ldots,\beta_{n}). 
\end{equation}
Hence, we get for all $i,j$ 
$$\mathbb{Q}(\zeta_{i,j,k}) \subset \mathbb{Q}(\beta_{1},\ldots,\beta_{n}).$$Since $\mathbb{Q}(\zeta_{i,j,k})/\mathbb{Q}$ is Galois, then $\text{Gal}(\mathbb{Q}(\beta_{1},\ldots,\beta_{n})/\mathbb{Q}(\zeta_{i,j,k}))$ is a normal subgroup of $\text{Gal}(\mathbb{Q}(\beta_{1},\ldots,\beta_{n})/\mathbb{Q})=\mathfrak{A}_{n}$, which is simple for $n\geq 5$. Hence, either $\mathbb{Q}(\beta_{1},\ldots,\beta_{n})=\mathbb{Q}(\zeta_{i,j,k})$, but this cannot be the case as $\mathfrak{A}_{n}$ is not abelian for $n\geq 4$, or $\Q=\mathbb{Q}(\zeta_{i,j,k})$. 

If $\mathbb{Q}(\zeta_{i,j,k}) = \mathbb{Q}$ for all $i,j$, since $\zeta_{i,j,k}$ is a root of unity, we have either $\zeta_{i,j,k}=1$ or $\zeta_{i,j,k}=-1$. If $\zeta_{i,j,k}=1$ for some $i \neq j$, by \eqref{argument galois correspondence}, we would have $\beta_{i}=\beta_{j}$, which is a contradiction. If $\zeta_{i,j,k}=-1$ for all $i \neq j$, then we have by \eqref{argument galois correspondence}, $\beta_{1}=-\beta_{2}=-\beta_{3}$, and so $\beta_{2}=\beta_{3}$, which is again a contradiction. Finally, for all non-zero integer $k$, we have $\beta^{k} \notin \mathbb{Q}$.

Therefore, applying the conclusion of Proposition \ref{prop pour montrer le point (1)} to \eqref{eq-alpha-sigma(alpha)}, there exists an integer $c$ such that $a_{k}-a_{\tau(k)}=c$ for all $k$. 

If $\ell$ if the order of $\tau$, we have, $ a_{1}=a_{\tau^{\ell}(1)}=\dots=a_{1}+c\ell$. So $c=0$ and $a_{\tau(k)}=a_{k}$ for all $k$. 

If the $a_{i}$ are all distinct, then $\tau$ is necessarily the identity. 

If $|\{a_{i}, 1 \le i \le n \}|= n-1$, we can suppose that $a_{1}=a_{2}$. We then have, $ \tau(k)=k$ for all $k\geq 3$ and $\tau(1),\tau(2) \in \{1,2\}$. 

But $\tau \in \Agothique_{n}$, so $\tau$ cannot be the transposition $(1 \; 2)$. Hence $\tau$ is the identity, concluding the proof.
\end{proof}

\begin{rmq}
    We notice that the argument from the previous proof allows us to remove the condition $\beta$ unit, also for \cite[Theorem 1.1 point (1)]{amoroso2018mahler}.
\end{rmq}

\subsection{Proof of Theorem \ref{thm 1.1}, point \eqref{thm1.1-it3}}\label{point2-thm1.1}
The goal of this section is to prove the following result, of which Theorem \ref{thm 1.1}, point \eqref{thm1.1-it3} is a simple corollary.
\begin{prop}\label{thm1.1-it2}
Let $\beta$ be an algebraic unit of degree $n \ge 5$, let $\beta_{1}, \ldots, \beta_{n}$ be its conjugates and assume that $\mathrm{Gal}(\mathbb{Q}(\beta_{1}, \ldots, \beta_{n})/\mathbb{Q}) = \mathfrak{A}_{n}$. Let $\alpha= \beta_{1}^{a_{1}} \ldots \beta_{n}^{a_{n}}$ where $a_{1}, \ldots, a_{n} \in \mathbb{Z}$. 

Let $y_{j}= a_{j} - \frac{1}{n}\sum_{i=1}^{n}a_{i}$ and set $|y|_1=\frac{1}{n}\sum_{j=1}^n |y_j|$. Then
$$M(\beta)^{|y|_{1}} \ge M(\alpha)^{{2}/{n!}} \ge M(\beta)^{c_n |y|_{1}},$$
where $c_{n}$ is an effective computable constant satisfying
\[\lim_{n\rightarrow \infty}{c_{n}} \sqrt{\frac{\pi n}{2}}=1.\]
\end{prop}

Before proceeding with the proof of this part, we 
need to recall some definitions from \cite[Section 2]{amoroso2018mahler} and adapt some other technical results therein.

Let \( H_{n} = \{ x \in \mathbb{R}^{n} \mid x_{1} + \dots + x_{n} = 0 \} \), and for \( x \in H_{n} \), we define the usual \( L^{1} \)-norm on \( \mathbb{R}^{n} \) as
$|x|_{1} = \frac{1}{n} \sum_{j=1}^{n} |x_{j}|.$

The group $\Agothique_{n}$ acts on $H_{n}$ by $\sigma(x)=\left(x_{\sigma(1)}, \dots, x_{\sigma(n)} \right)$.
For $x,y \in H_{n}$, we set
\begin{equation*}
s_{n}(x,y)=\frac{2}{n!}\sum_{\sigma \in \Agothique_{n}} \left|\frac{1}{n}\sum_{j=1}^{n}x_{\sigma(j)}y_{j} \right|.
\end{equation*}

We notice that $s_{n}$ is symmetric.

The proof of Proposition \ref{thm1.1-it2} is then a combination of three results. The first, given in the following proposition, links the Mahler measure of $\alpha$ to the function $s_n$.
\begin{prop}\label{Mahler-sn} Let $\alpha,\beta_1,\ldots,\beta_n$ and $y=(y_1,\ldots,y_n)$ be as in the statement of Proposition \ref{thm1.1-it2}. Let $x=(x_1,\ldots,x_n)$ where $x_i=\log|\beta_{i}|$.  
Then 
\[\log(M(\beta))=\frac{n}{2} |x|_1\]
and 
\[\log(M(\alpha))=\frac{ n (n!)}{4}  s_{n}(x,y) .\]
\end{prop}
\begin{proof}
Notice that
\[
\sum_{j=1}^{n}x_j=\log(N_{K/\Q}(\beta))=0
\]
so $x \in H_{n}$. Let $\mathcal{I}=\{j\in\{1,\ldots,n\}\mid |\beta_j|\geq 1\}$. Then 
\begin{align*}
    \sum_{j=1}^{n}| x_j|&= \sum_{j\in \mathcal{I}} \log|\beta_{j}|-\sum_{j\not\in \mathcal{I}} \log|\beta_{j}|= 2\sum_{j\in \mathcal{I}} \log|\beta_{j}|\\
    &=2\sum_{j=1}^{n}\log(\max(|\beta_{j}|,1))
    =2\log(M(\beta))
\end{align*}
proving the first equality in the statement.
By applying this equality to $\alpha$, whose conjugates are $\beta_{\sigma(1)}^{a_{1}}\dots \beta_{\sigma(n)}^{a_{n}}$ for $\sigma \in \Agothique_{n}$, we obtain 
\begin{align*}
    2 \log(M(\alpha))&= \sum_{\sigma \in \Agothique_{n}} |\log(|\beta_{\sigma(1)}^{a_{1}}\dots \beta_{\sigma(n)}^{a_{n}}|)|= \sum_{\sigma \in \Agothique_{n}} \left | \sum_{j=1}^{n}a_{j}x_{\sigma(j)} \right | . 
\end{align*}
Thus, by recalling that $a_{j}=y_{j}+\frac{1}{n} \sum_{i=1}^{n}a_{i}$ and that $\sum_{j=1}^n x_{\sigma(j)}=0$ for every $\sigma\in \Agothique_{n}$, we obtain
\begin{align*}
   2\log(M(\alpha))&
    =\sum_{\sigma \in \Agothique_{n}} \left | \sum_{j=1}^{n}y_{j}x_{\sigma(j)}+ \left( \frac{1}{n}\sum_{i=1}^{n}a_{i} \right) \left( \sum_{j=1}^{n}x_{\sigma(j)} \right) \right |\\
    &=\sum_{\sigma \in \Agothique_{n}} \left | \sum_{j=1}^{n}y_{j}x_{\sigma(j)} \right |
    =s_{n}(x,y)\cdot \frac{n (n!)}{2}.
\end{align*}
\end{proof}

The next step in the proof is to provide an upper and lower bound for the quantity $s_n(x,y)$. To this aim, for $h=1, \dots, n-1$, we define the vector $z^{(n,h)} \in H_{n}$ by
\[
z_j^{(n,h)} = \left\{
    \begin{array}{ll}
        \frac{n}{2h} & \mbox{if } j\leq h \\
        -\frac{n}{2(n-h)} & \mbox{if } j>h
    \end{array}
\right.
\]
and we set \[c_{n}=\underset{0<h,k<n}{\min}s_{n}(z^{(n,h)},z^{(n,k)}).\] 

We have the following lemma:

\begin{Lemma} \label{formule des classes}
    Let $n \ge 5$ and $A=\{1,\ldots,h\}$ for some $1 \le h \le n$. Let $G$ be the subgroup of $\Agothique_{n}$ defined as $G = \{\sigma \in \Agothique_{n} \mid \sigma(A) = A\}$. Then, for every $x \in H_{n}$, we have
    \begin{equation*}
        |G|^{-1}\sum_{\sigma \in G} \sigma(x) = z^{(n, h)}.
    \end{equation*}
\end{Lemma}

\begin{proof}
Let $x\in H_{n}$.
Using the fact that $x_{1} + \dots + x_{n} = 0$, we have
    \begin{align*}
        1 = |x|_{1} &= \frac{1}{n}\sum_{j=1}^{n}|x_{j}| = \frac{1}{n} \left( \sum_{j=1}^{h}x_{j} - \sum_{j=h+1}^{n}x_{j} \right) = \frac{2}{n} \sum_{j=1}^{h}x_{j} = -\frac{2}{n}\sum_{j=h+1}^{n}x_{j}.
    \end{align*}
    So, in particular
    \begin{equation}\label{sum-x_j}
    \sum_{j=1}^{h}x_{j} = \frac{n}{2}= -\sum_{j=h+1}^{n}x_{j} .
    \end{equation}
We also notice that $G$ restricted to $A$ is a transitive subgroup of $\mathfrak{S}(A) \simeq \mathfrak{S}_{h}$. Indeed, if $h=1$ there is nothing to show. If $h=2$, then $A=\{1,2\}$, and $\tau := (1 \; 2)(3 \; 4) \in G$ verifies $\tau(1)=2$. If $h \geq 3$, let us take $i,j,k \in A$, all distinct. Then $\tau := (i \; j \; k) \in G$ verifies $\tau(i)=j$. We set for $j,i \in A$, $G_{j,i} = \{\sigma \in G \mid \sigma(i)=j \}$. We denote the $i$-th component of the vector $\sum_{\sigma \in G}\sigma(x)$ by 
$$\left(\sum_{\sigma \in G } \sigma(x)\right)_{i}.$$ 
So, by Lemma \ref{lemme élémentaire transitif} and \eqref{sum-x_j}, we get for $1\leq i\leq h$
\[
\left(\sum_{\sigma \in G } \sigma(x)\right)_{i} = \sum_{\sigma \in G}x_{\sigma(i)}= \sum_{j=1}^{h} |G_{j,i}|x_{j} = \frac{|G|}{h} \sum_{j=1}^{h}x_{j}  = \frac{n}{2h}|G|.
\]
A similar argument shows that, for $h+1 \leq i \leq n$, we have
\[
\left(\sum_{\sigma \in G } \sigma(x)\right)_{i} = \frac{|G|}{n-h} \sum_{j=h+1}^{n}x_{j}  = -\frac{n}{2(n-h)}|G|,
\]
completing the proof of the lemma.
\end{proof}

From this lemma, we have the following proposition:
\begin{prop}\label{(2.1)}
For every $x,y\in H_n$ we have 
\[c_{n} \le \frac{s_{n}(x,y)}{|x|_{1}|y|_{1}} \le 1.\]
\end{prop}
\begin{proof}
We first notice that the upper bound is easy to obtain. We have
\begin{equation}\label{up-sn}
    s_n(x,y) \le \frac{2}{n!} \frac{1}{n}\sum_{\sigma \in \Agothique_{n}}\sum_{j=1}^{n}|x_{\sigma(j)}| |y_{j}| =\frac{2}{n!} \frac{1}{n} \sum_{j=1}^{n}|y_{j}|\sum_{\sigma \in \Agothique_{n}}|x_{\sigma(j)}|.
\end{equation}
By Lemma \ref{lemme élémentaire transitif}, since $\Agothique_{n}$ is a transitive subgroup of $\Sgothique_{n}$, we know that for all $1\le i,j \le n$ we have $$| \{\sigma \in \Agothique_{n} \; |\; \sigma(j)=i \}|=\frac{|\Agothique_{n}|}{n}=\frac{(n-1)!}{2}.$$ 
Therefore, replacing in \eqref{up-sn} we have
\begin{align*}
   \frac{2}{n!} \frac{1}{n} \sum_{j=1}^{n}|y_{j}| \sum_{\sigma \in \Agothique_{n}}|x_{\sigma(j)}| &= \frac{2}{n!} \frac{1}{n} \sum_{j=1}^{n}|y_{j}| \frac{(n-1)!}{2}\sum_{i=1}^{n}|x_{i}|\\
   &= \frac{1}{n^2}\sum_{j=1}^{n}|y_{j}| \sum_{i=1}^{n}|x_{i}| =|x|_{1}|y|_{1}
\end{align*}
concluding the proof of the upper bound.

We are now left with the proof of the lower bound. Let $x$ and $y$ be two non-zero vectors in  $H_{n}$ and let $h$ and $k$ be the number of non-negative components of  $x$ and $y$, respectively.
The proposition reduces to
    \begin{equation}\label{up-low-sn}
    \frac{s_{n}(x, y)}{|x|_{1}|y|_{1}} \ge \frac{s_{n}(z^{(n, h)}, y)}{|y|_{1}} \ge s_{n}(z^{(n, h)}, z^{(n, k)}).
    \end{equation}

Let $x=(x_1,\ldots,x_n)$ and $y=(y_1,\ldots,y_n)$.
Since for real numbers $c,c'>0$ we have $s_{n}(cx, c'y)=cc' s_{n}(x,y)$, we can suppose that 
$|x|_{1} = |y|_{1}=1$. 
Moreover, letting \[A = \{j \in \{1, \dots, n\}\mid x_{j} \ge 0\},\] since the function $\varphi_{n,y}(x)=s_{n}(x,y)$ is invariant under the action of $\Agothique_{n}$, we can further suppose that $A = \{1, \dots, h\}$. Let $G$ be the subgroup of $\Agothique_{n}$  defined as $G = \{\sigma \in \Agothique_{n}\mid  \sigma(A) = A\}$.

    As in \cite{amoroso2018mahler}, the proof now concludes with a convexity argument. As the function $\varphi_{n,y}(x)=s_{n}(x, y)$ is convex and recalling that $|x|_1=1$, we have by Lemma \ref{formule des classes}
    \[
    s_{n}(z^{(n, h)}, y) = s_{n} \left(  \frac{1}{|G|} \sum_{\sigma \in G} \sigma(x) , y \right) \leq \sum_{\sigma \in G} \frac{1}{|G|} s_{n}(\sigma(x), y) = \frac{s_{n}(x, y)}{|x|_1}.
    \]
    From this, using the symmetry of $s_{n}$ by inverting the roles of $x$ and $y$, and recalling that $|y|_1=|x|_{1}=1$, we get
    \[
    \frac{s_{n}(z^{(n, h)}, y)}{|y|_{1}} = \frac{s_{n}(y,z^{(n, h)})}{|y|_{1}} \geq s_{n}(z^{(n, k)}, z^{(n, h)})=s_{n}(z^{(n, h)}, z^{(n, k)})
    \]
which concludes the proof of \eqref{up-low-sn} and of Proposition \ref{(2.1)}.
\end{proof}

The next and last step to prove Proposition $\ref{thm1.1-it2}$ is to provide an asymptotic estimate for the quantity $c_n$. This is done in the following proposition.
\begin{prop} \label{(2.2)}
\[\lim_{n\rightarrow \infty}{c_{n}} \sqrt{\frac{\pi n}{2}}=1.\]
\end{prop}
The proof of Proposition \ref{(2.2)} relies on the following two lemmas. The first gives a closed formula for the quantity $s_{n}(z^{(n,h)},z^{(n,k)})$ which equals the one in \cite[Lemma 2.3]{amoroso2018mahler}, even though, we notice that the definition of the function $s_n$ therein is slightly different from ours.
  \begin{Lemma}\label{Lemma grosse égalité}
    Let $y \in H_{n}$ and $h,k \in \{1,\dots,n-1\}$. Then 
    \begin{equation}\label{form-sn-1}
    s_{n}(z^{(n,h)},y) = \frac{n}{2h(n-h)} \binom{n}{h}^{-1} \sum_{\substack{ S \subset \{1,\dots,n\}\\|S|=h }} \left | \sum_{j \in S} y_{j} \right|
    \end{equation}
    and 
    \begin{equation}\label{form-sn-2}
    s_{n}(z^{(n,h)},z^{(n,k)}) = \frac{n^{2}(h-\lfloor \frac{hk}{n} \rfloor )(k- \lfloor \frac{hk}{n} \rfloor)}{2hk(n-h)(n-k)} \binom{n}{h}^{-1} \binom{k}{ \lfloor  \frac{hk}{n} \rfloor} \binom{n-k}{h- \lfloor  \frac{hk}{n} \rfloor}
    \end{equation}
    where for a real number $x$, $\lfloor x \rfloor$ denotes as usual the greatest integer less than or equal to $x$.
\end{Lemma}

\begin{proof} 
    By the definition of $s_n$ and $z^{(n,h)}$, and recalling that, for $y=(y_1,\ldots,y_n)\in H_n$, we also have $\sigma(y)\in H_n$, which is equivalent to $\sum_{j=1}^n y_{\sigma(j)}=0$ for all $\sigma \in \mathfrak{A}_{n}$, we have
    \begin{align*}
        s_{n}(z^{(n,h)},y) &= \frac{2}{n!}\sum_{\sigma \in \Agothique_{n}} \left | \frac{1}{n} \sum_{j=1}^{n}z_{j}^{(n,h)}y_{\sigma(j)} \right |\\
        &= \frac{2}{n!}\sum_{\sigma \in \Agothique_{n} } \left | \frac{1}{2h} \sum_{j=1}^{h}y_{\sigma(j)}- \frac{1}{2(n-h)}\sum_{j=h+1}^{n}y_{\sigma(j)} \right |\\
        &= \frac{2}{n!}\sum_{\sigma \in \Agothique_{n} } \left | \left(\frac{1}{2h} + \frac{1}{2(n-h)} \right) \sum_{j=1}^{h}y_{\sigma(j)} \right | \\
        &= \frac{n}{h(n-h)}\frac{1}{n!} \sum_{\sigma \in \Agothique_{n}} \left | \sum_{j=1}^{h}y_{\sigma(j)} \right |.
    \end{align*}

    To conclude the proof of \eqref{form-sn-1} we need to show that 
    \[
    \frac{n}{h(n-h)}\frac{1}{n!} \sum_{\sigma \in \Agothique_{n}} \left | \sum_{j=1}^{h}y_{\sigma(j)} \right | = \frac{n}{2h(n-h)} \binom{n}{h}^{-1} \sum_{\substack{ S \subset \{1,\dots,n\}\\|S|=h }} \left | \sum_{j \in S} y_{j} \right| .
    \]

    Let $S \subset \{1,\dots,n \}$ such that $|S|=h$. Let \[G_{S}:=\{ \sigma \in \Agothique_{n}\mid  \sigma(\{1,\ldots,h\})=S \}\] and \[G'_{S}:=\{ \sigma \in \Sgothique_{n}\mid \sigma(\{1,\ldots,h\})=S \} \simeq \Sgothique_{h} \times \Sgothique_{n-h}.\] We suppose first that $h \le n-2$. If $\tau$ is the transposition switching $n$ with $n-1$, the map $\sigma \mapsto \sigma \circ \tau$ is a bijection between the sets $
    G_{S}$ and $G'_{S} \setminus G_{S}$.
In particular \[|G_{S}|=\frac{|G'_{S}|}{2}=\frac{h!(n-h)!}{2}\] providing the sought-for equality. If $h=n-1$, we use the same argument, but with $\tau=(1 \; 2)$ this time.

    The proof of \eqref{form-sn-2} now is precisely the one done in the proof of \cite[Lemma 2.3]{amoroso2018mahler}.
\end{proof}

We recall the following result which gives the asymptotic estimate for \eqref{form-sn-2}.
\begin{Lemma}[{\cite[Lemma 2.4]{amoroso2018mahler}}]\label{Lemma 1.4}
    Let $(n_{m})_{m}$, $(h_{m})_{m}$, $(k_{m})_{m}$ be sequences of real numbers satisfying  $0< h_{m},k_m<n_m$ and such that
    \[
    \underset{m \rightarrow +\infty}{\lim}n_{m}=+\infty, \quad \underset{m \rightarrow +\infty}{\lim} \frac{h_{m}}{n_{m}}=u, \quad \underset{m \rightarrow +\infty}{\lim} \frac{k_{m}}{n_{m}}=v
    \]
    for some $u,v \in [0,1]$. Then (with the convention $1/0=+\infty$), we have
    \[
    \underset{m \rightarrow +\infty}{\lim} 2 \sqrt{n_{m}} \cdot s_{n_{m}} \left(z^{(n_{m},h_{m})},z^{(n_{m},k_{m})} \right)=\frac{1}{\sqrt{2 \pi uv(1-u)(1-v)}}
    .\]

\end{Lemma}
We can now prove the  estimate for $c_n$ from Proposition \ref{(2.2)}.
\begin{proof}[Proof of Proposition \ref{(2.2)}]
For an integer $m\geq 1$, let $h_m$ and $k_m$ be the indices where the minimum is attained in the definition of $c_m$. Applying Lemma \ref{Lemma 1.4}  to the sequences $(h_m)_m$, $(k_m)_m$ and $(n_m)_m=(m)_m$, we obtain
\[\underset{m \rightarrow +\infty}{\lim} 2{\sqrt{2 \pi uv(1-u)(1-v)m}} \cdot s_{m}(z^{(m,h_{m}},z^{(m,k_{m})}))=1\] for some $u,v\in[0,1]$,
and using the fact that
    \[
    \underset{0<u,v<1}{\max}uv(1-u)(1-v)=\frac{1}{16}
    \]
we get the asymptotic estimate  \eqref{(2.2)}. 
\end{proof}

We have now collected all the results needed to prove Proposition \ref{thm1.1-it2}.
  \begin{proof}[Proof of Proposition \ref{thm1.1-it2}]

Combining Proposition \ref{Mahler-sn} and Proposition \ref{(2.1)} we get
\[
c_{n} \frac{2\log(M(\beta))}{n} |y|_{1} \le \frac{1}{n} \left(\frac{4}{n!} \log(M(\alpha))\right) \le \frac{2\log(M(\beta))}{n} |y|_{1}
\]
or equivalently
\[
\log(M(\beta)^{c_{n}|y|_{1}}) \le \log(M(\alpha)^{2/n!}) \le \log(M(\beta)^{|y|_{1}}).
\]
The result follows by taking the exponential and using Proposition  \ref{(2.2)}. 
\end{proof}

To conclude the proof of Theorem \ref{thm 1.1}, we  need the following lemma:
\begin{Lemma} \label{Lemma 1}
    Let $n\geq 5$ and let $ y=(y_1,\ldots,y_n) \in H_{n}$ be such that  $y_{j+1}-y_{j} \ge 1$ for all $j=1,\dots,n-2$. 
Then 
    \begin{enumerate}
        \item\label{item1-lemma1} If $y_{n}-y_{n-1} \ge 1$, then 
        $|y|_{1} \ge \frac{n-2}{4}.$
        \item\label{item2-lemma1} If $y_{n-1}=y_{n}$, then 
        $|y|_{1} \ge \frac{n-3}{5}.$
    \end{enumerate}
\end{Lemma}
\begin{proof}
    The proof of \eqref{item1-lemma1} is given in \cite[Lemma on p. 1614]{amoroso2018mahler}.
    For the proof of point \eqref{item2-lemma1}, let $1\leq k \le n-1$ be such that $y_{k} \le 0 < y_{k+1}$. 
Notice that  $k\leq n-2$ as $y_{n-1}=y_n$ by hypothesis.

We then have, $y_j \leq y_k-(k-j)\leq -(k-j)$ for $1\leq j\leq k$ and $y_j\geq y_{k+1}+(j-k-1)\geq j-k-1$ for $k+1\leq j\leq n-1$. So
    \begin{align*}
        n|y|_{1} &=\sum_{j=1}^{k}(-y_{j})+\sum_{j=k+1}^{n-1}y_{j}+|y_{n}|\ge \sum_{h=0}^{k-1}h+\sum_{h=0}^{n-k-2}h \\
        &= \frac{(k-1)k}{2}+\frac{(n-k-2)(n-k-1)}{2}\ge \frac{(n-1)(n-3)}{4}.
    \end{align*}
\end{proof}

We can finally conclude the proof of Theorem \ref{thm 1.1}.
\begin{proof}[Proof of point \eqref{thm1.1-it3} of Theorem \ref{thm 1.1}]
 If $\alpha=\beta_{1}^{a_{1}}\dots \beta_{n}^{a_{n}}$ is a generator of $\rm{Gal}(\Q(\beta_{1}, \dots, \beta_{n})/\Q)$ then,  by point \eqref{thm1.1-it1} of Theorem \ref{thm 1.1}, we have \[|\{a_{i}\mid  1 \le i \le n \}| \ge n-1.\] If the $a_{i}$'s are all distinct, then so are the $y_{i}$'s defined as \[y_{j}= a_{j} - \frac{1}{n}\sum_{i=1}^{n}a_{i}.\] 
As all $a_{i}$'s are integers, $|y_{\ell}-y_k|\geq 1$ for $\ell\neq k$ and moreover, since for every $\sigma\in \Sgothique_{n}$ we have $|\sigma(y)|_{1}=|y|_{1}$, up to reordering the terms we can suppose that $y_{j+1}-y_{j} \ge 1$ for all $1\leq j\leq n-1$.
Thus by  Lemma \ref{Lemma 1} \eqref{item1-lemma1} we have $|y|_{1} \ge \frac{n-2}{4}$.

If two $a_{i}$'s are equal, so are the corresponding $y_{i}$'s. Reasoning as before, up to reordering the $a_i$'s we might assume that $y_{j+1}-y_{j} \ge 1$ for $j=1,\dots,n-2$ and $y_{n-1}=y_n$. So, by Lemma \ref{Lemma 1} \eqref{item2-lemma1} we have $|y|_{1} \ge \frac{n-3}{5}$.

    In all cases, we have $|y|_{1} \ge \frac{n-3}{5}$, and therefore, using also Proposition \ref{(2.2)}, we get the asymptotic, as $n$ tends to infinity
    \[
    c_{n}|y|_{1} \sim \sqrt{\frac{2}{\pi n}}|y|_{1} \ge \sqrt{\frac{2}{\pi n}} \left(\frac{n-3}{5}\right) \sim \sqrt{\frac{2n}{25 \pi}}.
    \]
    This, together with Proposition \ref{thm1.1-it2}  gives 
    \[
    M(\alpha)^{\frac{2}{n!}} \ge M(\beta)^{(1+o(1))\sqrt{\frac{2n}{25 \pi}}}
    \]
    and taking logarithms we have
    \[
    h(\alpha) \ge (1+o(1))\sqrt{\frac{2n}{25 \pi}}\log(M(\beta)).
    \]

    By the explicit version of Dobrowolski's theorem  by Voutier \cite[Theorem, p. 83]{voutier}, we have 
    \[
    \log(M(\beta)) \ge \frac{1}{4}\left( \frac{\log(\log(n))}{\log(n)} \right)^{3}
    \]
concluding the  proof.
\end{proof}
\section{Proof of Theorem \ref{thm 1.2} } \label{Générateur de An version additive}
The aim of this section is to prove Theorem \ref{thm 1.2}, which is the analogue of Theorem \ref{thm 1.1} for \emph{additive} generators, that is generators obtained as linear combinations of conjugates. 

As usual, if $\beta_1,\ldots,\beta_n\in \overline{\Q}$ is a full set of Galois conjugates, when writing $\Agothique_{n} \subseteq \mathrm{Gal}(\Q(\beta_{1},\dots,\beta_{n})/\Q)$, we identify $\sigma\in \Agothique_{n}$ with the automorphism $\sigma(\beta_i)=\beta_{\sigma(i)}$. As before, we split the proof in different parts for clarity.
\subsection{Proof of Theorem \ref{thm 1.2}, point  \eqref{thm 1.2-it1}} \label{point-1-theorem1.3}

The proof of this point is the same, in additive notation, as the one carried in Section \ref{sec-proof-gen-mult} and it is even simpler than its multiplicative variant.
We begin with the following additive version of Proposition \ref{prop pour montrer le point (1)}.

\begin{prop} \label{prop version additive}
Let $\beta$ be an algebraic number of degree $n\geq 5$ with conjugates $\beta_{1},\ldots,\beta_{n}$ over $\Q$.  Suppose that $\Agothique_{n} \subset \mathrm{Gal}(\mathbb{Q}(\beta_{1},\ldots,\beta_{n})/\mathbb{Q})$. Then, for any integers $v_{1},\ldots,v_{n}$, not all equal, we have $$\sum_{i=1}^{n}v_{i}\beta_{i} \not\in \Q.$$
\end{prop}
\begin{proof}
    The proof is similar to that of Proposition \ref{prop pour montrer le point (1)},  and we identify the elements of $\mathrm{Gal}(\mathbb{Q}(\beta_{1},\ldots,\beta_{n})/\mathbb{Q})$  with permutations in the usual way. 

Assume  that $v_1\neq v_2$ and that there exists $r\in \Q$ such that \begin{equation*}\sum_{i=1}^{n}v_{i}\beta_{i}=r.\end{equation*} 
Applying the cycle $(1 \; 3 \; 2)$ to both sides of this equality  we obtain ${(v_1-v_2)}\beta_1+{(v_2-v_3)}\beta_2+{(v_3-v_1)}\beta_3=1$.
Applying  to this the cycle $(1 \; i \; j)$, for $i\neq j$ and $i,j\geq 4$, we obtain $(v_1-v_2)\beta_1=({v_1-v_2})\beta_i$ for all $i\geq 4$. Letting act on both sides of this equality 
 the cycles $(1 \; 2 \; k)$ and $(1 \; 3 \; \ell)$ for $k,\ell\not\in\{2,3,i\}$, we  obtain  $({v_{1}-v_{2}})\beta_1=({v_{1}-v_{2}})\beta_i$ for all $1\leq i\leq n$. 
So $\beta_{1}=\ldots=\beta_{n}$ and $\beta \in \mathbb{Q}$, which contradicts  the hypothesis that $\Agothique_{n} \subset \mathrm{Gal}(\mathbb{Q}(\beta_{1},\ldots,\beta_{n})/\mathbb{Q})$.
\end{proof}
We can now conclude the proof of the first part of Theorem \ref{thm 1.2}.
\begin{proof}[Proof of point $(1)$ of Theorem \ref{thm 1.2}]
The proof follows exactly  the one of point \eqref{thm1.1-it1} of Theorem \ref{thm 1.1}, by using the additive notation instead of the multiplicative one. 

Suppose that $\alpha=a_1\beta_1+\ldots+a_n\beta_n$ is such that $\Q(\alpha)=\mathbb{Q}(\beta_{1},\ldots,\beta_{n})$. If $|\{a_i\mid 1\leq n\leq i\}|<n-1$ then one can show that there exists a non-trivial $\sigma\in \mathrm{Gal}(\mathbb{Q}(\beta_{1},\ldots,\beta_{n})/\mathbb{Q})$ such that $\sigma(\alpha)=\alpha$ (we take $\sigma=(i \; j \; k)$ if $a_i=a_j=a_k$ or $\sigma=(i \; j)(k \; \ell)$ if $a_i=a_j$ and $a_k=a_{\ell}$).

For the converse, assume that $|\{a_{i}\mid 1 \le i \le n \}| \geq n-1$, let $\sigma\in \mathrm{Gal}(\Q(\beta_{1},...,\beta_{n})/ \Q(\alpha))$ and $\tau=\sigma^{-1}$. 
The equality $\sigma(\alpha)=\alpha$ gives  $\sum_{k=1}^{n}{(a_{k}-a_{\tau(k)})}\beta_{k}=0$.
By Proposition \ref{prop version additive}, there exists an integer $c$ such that $a_{k}-a_{\tau(k)}=c$ for all $k$. 
If $\ell$ if the order of $\tau$, we have, $ a_{1}=a_{1}+c\ell$, so $c=0$ and $a_{\tau(k)}=a_{k}$ for all $k$. 
As $|\{a_{i}, 1 \le i \le n \}|\geq n-1$, we have that $\tau$ is either the identity or a transposition, but this last case cannot occur as $\tau\in \Agothique_{n}$.
\end{proof}

\subsection{Proof of Theorem \ref{thm 1.2}, point  \eqref{thm 1.2-it2}}\label{sec-point(2)-thm1.3}

This section deals with the proof of the last part of Theorem \ref{thm 1.2}. 

First, as in the multiplicative case, we notice that one can easily establish the following  upper bound for the height of $\alpha$: 
\begin{prop}\label{upp-b-m(alpha)-additive} Let $\beta$ be an algebraic integer of degree $n \ge 5$, let $\beta_{1}, \ldots, \beta_{n}$ be its conjugates and suppose that $\mathrm{Gal}(\mathbb{Q}(\beta_{1}, \ldots, \beta_{n})/ \mathbb{Q}) = \mathfrak{A}_{n}$. Let $a_1,\ldots,a_n$ be integers and suppose that
 $\alpha=a_1 \beta_1+\ldots+a_n\beta_n$ is a generator of $\mathbb{Q}(\beta_{1}, \ldots, \beta_{n})/ \mathbb{Q}$. Then 
    $$\log\left(M(\beta)\sum_{i=1}^{n}|a_{i}|\right) \ge h(\alpha).$$
\end{prop}
\begin{proof}
    It suffices to observe that, for every $\sigma \in \Agothique_{n}$, one has
    \begin{align*}
         |\sigma(\alpha)| &= \left|\sum_{i=1}^{n}a_{i}\beta_{\sigma(i)} \right| 
       \le \max_{1\le i\le n} |\beta_{i}| \sum_{i=1}^{n}|a_{i}| 
        \le M(\beta) \sum_{i=1}^{n}|a_{i}|.
    \end{align*}
Notice that, since $\beta$ is an algebraic integer, we have  \[M(\alpha)\le \max_{\sigma\in  \Agothique_{n}}|\sigma(\alpha)|^{\frac{n!}{2}},\]
and taking logarithms, this allows us to conclude.
\end{proof}

The proof of the lower bound in Theorem \ref{thm 1.2}, point  \eqref{thm 1.2-it2}, is instead  more delicate. Before starting it we fix the setting and introduce some notation.  
Let $\beta_1,\ldots,\beta_n$ be algebraic integers of degree $n\geq 5$ and let $a_{1},\dots,a_{n}$ be integers. Let 
\begin{equation}\label{form-alpha}
\alpha=\sum_{i=1}^{n}a_{i}\beta_{i}
\end{equation}
and  assume from this point onward that $\Q(\alpha)=\mathbb{Q}(\beta_{1},\ldots,\beta_{n})$ and that $\mathrm{Gal}(\mathbb{Q}(\beta_{1},\ldots,\beta_{n})/ \mathbb{Q})=\Agothique_{n}$.

We now recall some  notation from \cite[Section 4]{amoroso2018mahler}. For $\tau \in \Sgothique_{n}$, we let
$$
\alpha_{\tau}=\sum_{i=1}^{n}a_{i}\beta_{\tau(i)},
$$
where $\Sgothique_{n}$ acts on $\mathbb{Q}(\beta_{1},\ldots,\beta_{n})$ as usual by $\sigma(\beta_{i})=  \beta_{\sigma(i)}$.

We observe that for $\sigma \in \Agothique_{n}$, we have $\sigma(\alpha)=\alpha_{\sigma}$. Also, for $\tau, \tau' \in \Sgothique_{n}$, we have 
\begin{equation} \label{Sn agit sur Q(beta)}
    \alpha_{\tau \tau'}=(\alpha_{\tau'})_{\tau}.
\end{equation}

One of the key ideas in the proof of  the last part of Theorem \ref{thm 1.2}, and one of the main novelties of our article, is the following lemma, which describes how the Mahler measure of $\alpha$ changes under the action of an element that \emph{should not act on $\alpha$}. More precisely, it compares  $M(\alpha)$ and $M(\alpha_{\tau})$, in the case where $\tau$ is a transposition, so, by hypothesis, an element extraneous to $\mathrm{Gal}(\Q(\alpha)/\Q)$.

\begin{Lemma} \label{M(tau(alpha)) en fonction de M(alpha)}
 Let $\tau \in \Sgothique_{n}$ be a transposition, then $$M(\alpha_{\tau}) \le 5^{n} M(\alpha)^{5}.$$
\end{Lemma}

\begin{proof}
   
For simplicity, we may suppose that $\tau=(1 \; 2)$. 

Let $\sigma=(1 \; 3 \; 2)$ and set $$\gamma= \alpha-\sigma(\alpha)=\beta_{1}(a_{1}-a_{2})+\beta_{2}(a_{2}-a_{3})+\beta_{3}(a_{3}-a_{1}).$$ If $\sigma'=(1 \; 4 \; 5)$ and $\delta=(4 \; 2)(3 \; 5)$, then $$\delta(\gamma-\sigma'(\gamma))=\delta((\beta_{1}-\beta_{4})(a_{1}-a_{2}))=(\beta_{1}-\beta_{2})(a_{1}-a_{2})=\alpha-\alpha_{\tau}.$$  

So we have \begin{equation}\label{alpha-tau-linear-comb}\alpha_{\tau}=\alpha-\delta(\alpha)+\delta \sigma(\alpha)+\delta \sigma' (\alpha)-\delta \sigma' \sigma (\alpha).
\end{equation} 

Notice that $\delta, \sigma, \sigma'\in \Agothique_{n}$ and so is their product, hence, by the above equality, $\alpha_{\tau}$ equals a linear combination of conjugates of $\alpha$. By elementary inequalities for the height and its invariance under Galois action, we obtain that $h(\alpha_{\tau}) \le 5h(\alpha)+\log{5}$. By using that $M(\alpha)=e^{n h(\alpha)}$, we have the sought result.
\end{proof}
We now use Lemma \ref{M(tau(alpha)) en fonction de M(alpha)} to get a lower bound for $M(\alpha)$ independent on the transposition $\tau$.
To this aim we need to introduce some more notation.

Let  $T_n$ be the set of transpositions of $\Sgothique_{n}$ and,
 for $\tau=(i ,j) \in T_{n}$, let $A_{\tau}$ be the set of permutations of $\Sgothique_{n}$ with support $\{1,\ldots,n\} \setminus \{i,j\}$ and without orbits of length 2. 

For $n \ge 3$, let $\Lambda_{n}$ be the set of permutations of $\Sgothique_{n}$ without orbits of length $1$ or $2$.
Let also $$\Delta= \prod_{\tau \in T_{n}} \prod_{\sigma \in A_{\tau}} |\alpha_{\sigma \tau}-\alpha_{\sigma}|$$
and, for a vector of integers $\mathfrak{a}=(a_1,\ldots,a_n)$, let
$$V(\mathfrak{a})=\prod_{1\le i<j\le n}(a_{j}-a_{i}).$$

For an algebraic integer $\beta$ of degree $n$ and conjugates $\beta_1,\ldots,\beta_n$ we also let \[\mathrm{disc}(\beta)=\mathrm{disc}(1,\beta,\beta^{2},\dots,\beta^{n-1})=\prod_{1\leq i<j\leq n} (\beta_i-\beta_j)^2.\]

We have the following result.
\begin{prop}\label{minoration de M(alpha)} Let $n\geq 5$, and let $\alpha=\sum_{i=1}^{n}a_{i}\beta_{i}$ be as in \eqref{form-alpha}. Set $\beta=\beta_1$ and $\mathfrak{a}=(a_1,\ldots,a_n)$. Then
\[
   M(\alpha) \ge 5^{-n/6}\left(2^{-|T_{n}|}|V(\mathfrak{a})||\mathrm{disc}(\beta)|^{1/2} \right)^{|\Lambda_{n-2}|/6}.  
\]
\end{prop}
\begin{proof}
For $\tau=(i \; j) \in T_{n}$ with $1 \le i<j \le n$ and $\sigma \in A_{\tau}$, we have $$\alpha_{\sigma \tau}-\alpha_{\sigma}=(a_{j}-a_{i})(\beta_{i}-\beta_{j}).$$

Thus,  \begin{equation} \label{Delta=....}
    \Delta= (|V(\mathfrak{a})||\mathrm{disc}(\beta)|^{1/2})^{|\Lambda_{n-2}|}
\end{equation} 
as for each $\tau\in T_n$, $|A_{\tau}|=|\Lambda_{n-2}|$.

We now want to upper bound $\Delta$. Clearly we have that
\begin{equation*}
\Delta \le \prod_{\tau \in T_{n}} \prod_{\sigma \in A_{\tau}} 2\max(|\alpha_{\sigma \tau}|,1) \max(|\alpha_{\sigma}|,1).
\end{equation*}

Let $\tau$,$\tau' \in T_{n}$, and let $\sigma \in A_{\tau}$, $\sigma'\in A_{\tau'}$. Then $ \sigma\tau \neq \sigma'$ (since they have different supports). If $\tau=\tau'$ and $\sigma \neq \sigma '$ then $\tau \sigma \neq \tau' \sigma'$. If $\tau \neq \tau'$ then $\sigma\tau \neq \sigma'\tau' $ (since $\sigma$ has no orbits of length $2$), and $\sigma \neq \sigma'$ (since the supports of $\sigma$ and $\sigma'$ are not the same).

Thus, from the last inequality we deduce 
\begin{equation*}
\Delta \le 2^{|T_{n}||\Lambda_{n-2}|} \prod_{\sigma \in \Sgothique_{n}}\max(|\alpha_{\sigma}|,1).\end{equation*}

Writing $\Sgothique_{n}$ as the semidirect product of $\Agothique_{n}$ and a transposition $\tau$, that we shall fix, from the last inequality we have
\begin{align*}
    \Delta &\le 2^{|T_{n}||\Lambda_{n-2}|} \prod_{\sigma \in \Agothique_{n}}\max(|\alpha_{\sigma}|,1) \prod_{\sigma \in \Agothique_{n}}\max(|\alpha_{ \tau \sigma}|,1) \\
    &= 2^{|T_{n}||\Lambda_{n-2}|} \prod_{\sigma \in \Agothique_{n}}\max(|\sigma(\alpha)|,1) \prod_{\sigma \in \Agothique_{n}}\max(|\sigma(\alpha_{\tau})|,1) 
\end{align*}
hence by \eqref{Sn agit sur Q(beta)} we get
\begin{equation*}
    \Delta  \le 2^{|T_{n}||\Lambda_{n-2}|}M(\alpha)M(\alpha_{\tau}). 
\end{equation*}

From Lemma \ref{M(tau(alpha)) en fonction de M(alpha)} and the above inequality, we have $$\Delta \le 2^{|T_{n}||\Lambda_{n-2}|}M(\alpha)M(\alpha_{\tau}) \le 2^{|T_{n}||\Lambda_{n-2}|} 5^{n}M(\alpha)^{6}. $$

By using the equality \eqref{Delta=....}, we obtain $$\left(|V(\mathfrak{a})||\mathrm{disc}(\beta)|^{1/2} \right)^{|\Lambda_{n-2}|} \le  2^{|T_{n}||\Lambda_{n-2}|} 5^{n}M(\alpha)^{6}$$
which proves the result.
\end{proof}

We now proceed to give estimate for the quantities involved in the bound from Proposition \ref{minoration de M(alpha)}.
The following lemma is an improvement of \cite[Lemma, p. 1615]{amoroso2018mahler}. 
\begin{Lemma}\label{bound-Gamma_n}
    For $n \ge 3$, we have $$|\Lambda_{n}| \ge \frac{n!}{8}.$$
\end{Lemma}

\begin{proof}
    
Let $d_{n}$ be the $n$-th derangement number, i.e., the number of permutations of $\Sgothique_{n}$ without fixed points. It is known that (\cite[Corollary on p. 163]{derangement}) $$d_{n}=n!\sum_{k=0}^{n}\frac{(-1)^{k}}{k!}.$$

    By the inclusion–exclusion principle, we have $$|\Lambda_{n}| \ge d_{n}-\frac{n(n-1)}{2}d_{n-2}.$$ Letting $$ f(n)=\frac{\left(d_{n}-\frac{n(n-1)}{2}d_{n-2} \right)}{n!},$$ we want to show that $ f(n)\geq f(4)={1}/{8}$ for all $n\geq 3$.  
    We can check it numerically for $n=3,4,5$. Therefore we may assume $n\geq 6$.
We have
\begin{equation*}
        f(n)= \frac{1}{2}\sum_{k=0}^{n-2}\frac{(-1)^{k}}{k!}+\frac{(-1)^{n-1}}{(n-1)!}+\frac{(-1)^{n}}{n!} 
        \ge \frac{1}{2}\sum_{k=0}^{n-2}\frac{(-1)^{k}}{k!}+\frac{1-n}{n!} 
    \end{equation*}
    where, for the last inequality, a proof by exhaustion based on the parity of $n$ is sufficient.
The study of alternating series gives $$\sum_{k=4}^{n-2}\frac{(-1)^{k}}{k!} \ge 0$$ or equivalently, $$\sum_{k=0}^{n-2}\frac{(-1)^{k}}{k!} \ge \sum_{k=0}^{3}\frac{(-1)^{k}}{k!}.$$
    Additionally, we observe that the sequence $(\frac{1-n}{n!})_{n \ge 2}$ is increasing. Thus for $n \ge 6$,  we have $$f(n) \ge  \frac{1}{2}\sum_{k=0}^{3}\frac{(-1)^{k}}{k!}+\frac{1-6}{6!}=\frac{23}{144} \ge \frac{1}{8} .$$ 
  
This concludes the proof.
\end{proof}

\begin{proof}[Proof of point $(2)$ of Theorem \ref{thm 1.2}]
Notice that the statement is trivial for $5\le n \le 9$ as $M(\alpha)$ is always bigger or equal than 1.
So, from now on, we can assume that $n\geq 10$.  

By Lemma \ref{bound-Gamma_n}, since $n-2 \ge 3$, we have  \begin{equation}\label{size-Gamma-n}|\Lambda_{n-2}| \ge \frac{(n-2)!}{8}.\end{equation}
Also, one has
\begin{equation}\label{bound-T_n}
 |T_{n}|=\frac{n(n-1)}{2} \text{ and } |\mathrm{disc(\beta)}| \ge 1.
\end{equation}

Now, in order  to use  Proposition \ref{minoration de M(alpha)} to lower bound $M(\alpha)$, we are left to estimate $|V(\mathfrak{a})|$. Notice that this quantity is zero precisely when the $a_i$'a are not all distinct, in which case  the estimate in Proposition \ref{minoration de M(alpha)} is trivial.
 We then carry on the proof distinguishing two cases.

\subsubsection*{Case (a): all $a_i$'s are distinct}
In this case, like in \cite[p. 1616]{amoroso2018mahler}, we are going to establish a lower bound for $|V(\mathfrak{a})|\neq 0$. Since $|V(\mathfrak{a})|$ is symmetric with respect to $a_{i}$, we may assume for simplicity that $a_{j+1} - a_{j} \geq 1$ for $j = 1, \dots, n-1$.
 Then, we have 

\begin{equation*}|V(\mathfrak{a})|=\prod_{i=1}^{n}\prod_{j=i+1}^{n}|a_{j}-a_{i}| \ge \prod_{i=1}^{n}\prod_{j=i+1}^{n}(j-i)=\prod_{h=1}^{n-1}h!.\end{equation*}

By  \cite[Lemma 1, p. 84]{voutier} we have
\begin{equation*}|V(\mathfrak{a})|\ge 
    \prod_{h=1}^{n-1}h! \ge \exp \left( \frac{n^{2} \log(n)}{2}- \frac{3n^{2}}{4} \right).
\end{equation*}

We notice that, as $n \ge 10$, we have $$2^{-\frac{n(n-1)}{2}} \exp \left( \frac{n^{2} \log(n)}{2}- \frac{3n^{2}}{4} \right) \ge 1.$$ 

Therefore, using this, \eqref{size-Gamma-n} and \eqref{bound-T_n}, 
 Proposition \ref{minoration de M(alpha)} gives
\begin{align*}
    M(\alpha)^{\frac{2}{n!}}
  &\ge \left( 5^{-\frac{n}{6}} \left(2^{-\frac{n(n-1)}{2}}\exp \left( \frac{n^{2} \log(n)}{2}- \frac{3n^{2}}{4} \right) \right)^{\frac{(n-2)!}{48}}\right)^{\frac{2}{n!}}\\
    &= 5^{-\frac{1}{3(n-1)!}} 2^{-\frac{1}{48}}\exp \left( \frac{n}{n-1} \left( \frac{ \log(n)}{48}- \frac{1}{32} \right) \right).
\end{align*}

Notice that, as $n \ge 10$, we have $ \frac{ \log(n)}{48}\ge \frac{1}{32}$ and so
\begin{equation*}
    \exp \left( \frac{n}{n-1} \left( \frac{ \log(n)}{48}- \frac{1}{32} \right) \right) \ge \exp \left(  \frac{ \log(n)}{48}- \frac{1}{32} \right).
\end{equation*}

Finally
\begin{align*}
    M(\alpha)^{\frac{2}{n!}} &\ge 5^{-\frac{1}{9!3}} 2^{-\frac{1}{48}}\exp \left( \frac{ \log(n)}{48}- \frac{1}{32}  \right) 
    \ge \left(\frac{n}{9} \right)^{\frac{1}{240}}.
\end{align*}

\subsubsection*{Case (b) : the $a_i$'s are not all distinct}

Now we suppose that two of the $a_{i}$'s are equal, for instance $a_{1}=a_{2}$.  We set $$r=2\max_{1\le i,j \le n}(|a_{j}-a_{i}|)$$
and  $$\alpha'=\alpha+r(\beta_{1}-\beta_{2}).$$

Clearly \[\alpha'=\sum_{i=1}^{n}a_{i}'\beta_{i}\] where  $a_{1}'=a_{1}+r \in \Z$, $a_{2}'=a_{2}-r \in \Z$, and $a_{i}'=a_{i} \in \Z$ for all $i\ge 3$.

Since $a_{1}=a_{2}$ and $\alpha$ is a generator of $\Q(\beta_{1},\dots,\beta_{n}) / \Q$, we know, by point \eqref{thm 1.2-it1} of Theorem \ref{thm 1.2}, that the $a_{i}$'s are all distinct for $3\le i \le n$. 

The same holds for the $a_{i}'$'s for $3 \le i \le n$. Of course, we have $a_{1}' \neq a_{2}'$. Then, by the choice of $r$, we also have $a_{1}',a_{2}' \notin \{a_{i}' \mid 3 \le i \le n \}$. 

Hence, again by point \eqref{thm 1.2-it1} of Theorem \ref{thm 1.2}, $\alpha'$ is a generator of $\Q(\beta_{1},\dots,\beta_{n}) / \Q$ and the integers $a_1',\ldots,a_n'$ are all distinct.

 Thus, according to  case (a) of the proof, we have 
\begin{equation}\label{M(alpha')}M(\alpha')^{\frac{2}{n!}} \ge 5^{-\frac{1}{9!3}}  2^{-\frac{1}{48}}n^{\frac{1}{48}} \exp \left(- \frac{1}{32} \right).\end{equation}

We now want to compare $M(\alpha)$ and $M(\alpha')$. We will  need the following result, which plays here the role of Lemma \ref{M(tau(alpha)) en fonction de M(alpha)}.
\begin{Lemma} \label{(a_i-a_j)(beta_i-beta_j)}
    Let $n\geq 5$ and $i,j,k,l \in \{1, \dots,n \}$. Then $$h(\alpha+2(a_{i}-a_{j})(\beta_{k}-\beta_{l})) \le 5h(\alpha)+\log(16).$$
\end{Lemma}

\begin{proof}
We may suppose $i \neq j$ and $k \neq l$ otherwise the result is trivial. Let $\tau=(i \; j)\in T_n$. 
Then $$\alpha-\alpha_{\tau}=(a_{i}-a_{j})(\beta_{i}-\beta_{j}).$$
By \eqref{alpha-tau-linear-comb} from Lemma \ref{M(tau(alpha)) en fonction de M(alpha)} we also have
    $$\alpha-\alpha_{\tau}=\sigma_{1}(\alpha)+\sigma_{2}(\alpha)-\sigma_{3}(\alpha)-\sigma_{4}(\alpha)$$
for some $\sigma_{1},\sigma_{2},\sigma_{3},\sigma_{4} \in \Agothique_{n}$.

 As $n\geq 5$, we can check that there is $\sigma \in \Agothique_{n}$ such that $\sigma(\beta_{i}-\beta_{j})=\beta_{k}-\beta_{l}$ (if $i\neq k$ and $j\neq \ell$ one can simply take $\sigma=(i \; k)(j \; \ell)$, while if $i=k$ and $j\neq \ell$ one can take $\sigma=(j,\ell)(s,t)$ for some $s\neq t$ and $s,t\not\in\{k=i,j,\ell\}$). 
By properties of the height, we have
    \begin{align*}
        h(\alpha+2(a_{i}-a_{j})(\beta_{k}-\beta_{l}))&=h(\alpha+2(\alpha-\alpha_{\tau}))\\
       &= h(\alpha+2(\sigma_{1}(\alpha)+\sigma_{2}(\alpha)-\sigma_{3}(\alpha)-\sigma_{4}(\alpha))) \\
   &\le 5h(\alpha)+\log(16).
    \end{align*}
\end{proof}
Returning to our proof, by Lemma \ref{(a_i-a_j)(beta_i-beta_j)} 
\begin{align*}
    h(\alpha')&=h\left(\alpha+2\max_{1\le i,j \le n}(|a_{j}-a_{i}|)(\beta_{1}-\beta_{2}) \right) \le 5h(\alpha)+\log(16)
\end{align*}
hence $M(\alpha') \le 16^{n}M(\alpha)^{5}$.
Using also \eqref{M(alpha')}, we get
$$16^{\frac{2}{(n-1)!}}M(\alpha)^{\frac{10}{n!}} \ge 5^{-\frac{1}{9!3}}  2^{-\frac{1}{48}}n^{\frac{1}{48}} \exp \left(- \frac{1}{32} \right).$$ 

Finally, recalling that $n \ge 10$, we get
\begin{align*}
    M(\alpha)^{\frac{2}{n!}} &\ge 16^{-\frac{2}{9!5}} 5^{-\frac{2}{9!30}} 
    2^{-\frac{1}{240}}n^{\frac{1}{240}} \exp \left(- \frac{1}{160} \right) \ge \left(\frac{n}{9} \right)^{\frac{1}{240}}.
\end{align*}

\end{proof}

\section{Some examples and final remarks}\label{section-examples} 
In this section we give some application of our results to explicit families of algebraic numbers generating Galois extensions of $\Q$ with groups $\mathfrak{A}_{n}$. We also discuss the adaptability of the proof strategy of our main results to generators of Galois extensions having group other than $\Agothique_n$ or $\Sgothique_n$.
\subsection{On certain Laguerre polynomials}
A family of particular interest for our applications is that of generalized Laguerre polynomials, defined, for $n \ge 1$ and $\alpha \in \mathbb{Q}$, as
\begin{equation*}
L_{n}^{(\alpha)}(x) = \sum_{j=0}^{n} \binom{n+\alpha}{n-j} \frac{(-x)^{j}}{j!} 
\end{equation*}
where
\begin{equation*}
\binom{n+\alpha}{n-j} = \frac{(n+\alpha)(n-1+\alpha) \dots (j+1+\alpha)}{(n-j)!j!}.
\end{equation*}
This family, first studied by  Schur, has received the attention of many authors (see for instance \cite{LaguerreExemple1}, \cite{LaguerreExemple2}, \cite{LaguerreExemple3}, \cite{LaguerreExemple4}, \cite{LaguerreExemple5}, and \cite{LaguerreExemple6} for some reference on the topic).

The following result considers a special family of Laguerre polynomials having Galois group $\Agothique_n$:
\begin{prop}\label{laguerre-ex}
Let $n\geq 1$ be an integer divisible by $4$ and  let $\beta_{n,1},\ldots,\beta_{n,n}$ be the roots of the polynomial  $$L_{n}^{(-n-1)}(x)=1+x+\frac{x^{2}}{2}+\dots+\frac{x^{n}}{n!}.$$ Let $a_1,\ldots,a_n$ be integers such that $|\{a_1,\ldots,a_n\}|\geq n-1$. Then:
\begin{enumerate}[(a)] 
\item\label{part-a} If  $\alpha_n=\beta_{n,1}^{a_1}\cdots\beta_{n,n}^{a_n}$ and $\sum_{i=0}^n a_i\neq 0$ one has that $\Q(\alpha_n)/\Q$ is a Galois extension with Galois group $\Agothique_n$ and
$h(\alpha_n)$ tends to infinity with $n$.
\item\label{part-b} 
If $\alpha'_n=a_1 \beta_{n,1}+\ldots+a_n \beta_{n,n}$, then $\Q(\alpha'_n)/\Q$ is a Galois extension with Galois group $\Agothique_n$ and $h(\alpha'_n) $ tends to infinity with $n$.
\end{enumerate}
\end{prop}

\begin{proof}
Schur showed in \cite{troncatureEXP} that $L_{n}^{(-n-1)}(x)$ is irreducible and its Galois group over $\mathbb{Q}$ is $\Agothique_n$ if $n$ is divisible by $4$ (while it is $\Sgothique_n$ otherwise).
Let $\beta_n$ be one of the roots of $L_{n}^{(-n-1)}(x)$.

We start by proving part \eqref{part-a}. By point \eqref{thm1.1-it1} of Theorem \ref{thm 1.1}, we know that $\Q(\alpha)$ is the splitting field of $L_{n}^{(-n-1)}(x)$ over $\Q$. This proves the statement on the Galois group.
Notice also that that $N_{\Q(\beta)/\Q}(\beta)=n!$, hence, by Theorem \ref{generalresult}, $$M(\alpha)^{1/[\Q(\alpha):\Q]} \ge (n!)^{\frac{1}{n}\left |\sum_{i=1}^{n}a_{i} \right |}$$
and, since $\sum_{i=1}^{n}a_{i} \neq 0$ 
we deduce the statement on $h(\alpha_n)$. 

To prove part \eqref{part-b}, we simply notice that all hypotheses of Theorem \ref{thm 1.2} are satisfied, and we can conclude straightforwardly. 
\end{proof}

\begin{rmq}
We remark that our Theorem \ref{thm 1.1} cannot be applied to any Laguerre polynomial, while  \cite[Theorems 1.1 and 1.2]{amoroso2018mahler} can. Indeed, to apply point \eqref{thm1.1-it3} of Theorem \ref{thm 1.1}, we need that $\beta$ is a unit. However the roots of $L_{n}^{(\alpha)}(x)$ are units if and only if $\alpha=0$, and it follows from  \cite{IrreductibilitePolynomeLaguerreGeneralise} and  \cite{criterePolynomeLaguerreGeneralise}
that, for all but finitely many integers $n$, $L_{n}^{(\alpha)}(x)$ is irreducible over $\mathbb{Q}$ and its Galois group is $\Sgothique_{n}$.
\end{rmq}

\subsection{On the adaptation of the strategy to other subgroups of $\Sgothique_{n}$} \label{section adaptation de la methode}
One might question whether the proofs of  Theorems \ref{thm 1.1} and  \ref{thm 1.2}, as well as their counterparts for $\Sgothique_{n}$
in \cite{amoroso2018mahler}, could be modified to identify generators of Galois extensions with different Galois groups.

\subsubsection*{On the proof of  point \eqref{thm1.1-it1}} We first notice that, in all the above cited results, the necessary and sufficient conditions on the $a_i$'s so that $\Q(\alpha)/\Q$ is a Galois extension of group $G\in \{\Sgothique_{n},\Agothique_{n}\}$ can be rephrased by saying that $|\{a_1,\ldots,a_n\}|\geq n-k+1$ where $k$ is the index of $G$ in $\Sgothique_{n}$.
It might be tempting to inquire if a similar condition holds for other subgroups $G$ of $\Sgothique_{n}$ of index $2<k\le n$, but no such groups exist. Indeed, by \cite[Chap. 2, \S 85. Corollary p.64]{Allan-Clark}, $\Agothique_{n}$ is the only subgroup of $\Sgothique_{n}$ of index strictly smaller than $n$, while by \cite[Proposition 8.10]{perrin}, for $n\neq 6$ all the subgroups of $\Sgothique_{n}$ of index $n$ are of the form $\{\sigma\in\Sgothique_{n}\mid \sigma(i)=i\}$, hence they are not transitive.

A second observation is that our proofs also use the fact that if $G$ is either $\mathfrak{S}_{n}$ or $\mathfrak{A}_{n}$, then $G$ contains all $k$-cycles for some integer $2\le k \le n$ (noting that $k = 2$ for $\mathfrak{S}_{n}$ and $k = 3$ for $\mathfrak{A}_{n}$). However, if a subgroup $G$ of $\Sgothique_{n}$ contains all $k$-cycles for some $k$, then it contains the subgroup generated by all of them, which is normal in $\Sgothique_{n}$ (as $k$-cycles form a full conjugacy class). But, for $n\geq 5$, this implies that this subgroup is either $\Sgothique_{n}$ or $\Agothique_{n}$ ($\Sgothique_{n}$ if $k$ is even or $\Agothique_{n}$ otherwise), so in both cases, $G$ contains $\mathfrak{A}_{n}$.

\subsubsection*{On the proof of  point \eqref{thm1.1-it3}}
We notice that, to prove point  \eqref{thm1.1-it3} of Theorem  \ref{thm 1.1}, one crucial ingredient is Proposition \ref{(2.1)}.
Indeed, the remainder of the results can essentially be derived using the same methodology, irrespective of the specific subgroup of $\mathfrak{S}_n$ under consideration, by suitably accounting for the pertinent quantities. 

One point used in the proof of Proposition \ref{(2.1)} is that, if $G$ is either $\Sgothique_n$ or $\Agothique_n$ then, for any $A\subseteq \{1,...,n \}$, the stabilizer $\mathrm{Stab}_{G}{(A)}=\{\sigma \in G\mid  \sigma(A)=A\}$ is a transitive subgroup of $\Sgothique(A)\simeq \Sgothique_{|A|}$. 
However this property does not hold for other subgroups $G$ of $\Sgothique_n$ when $n$ is big enough.

Indeed, notice first that any transitive subgroup of $\Sgothique_{|A|}$ has a cardinality divisible by $|A|$, so $|G|$ must be divisible by $\mathrm{lcm}(1,\dots,n)$. In particular $|G|\ge 2^{n-1}$.

Moreover, it is easy to check that $G$ must be a $2$-transitive subgroup of $\Sgothique_{n}$. Indeed, let $(x_{1}, y_1), (x_{2},y_{2}) \in \{1,\dots,n\}^2$ be such that $x_{1} \ne y_{1}$ and $x_{2} \ne y_{2}$. Then there exists $g\in G$ such that $g(x_1)=x_2$ and $g(y_1)=y_2$ where
\begin{itemize}
\item[-] if $y_{1} \neq x_{2}$, we can take $g=\gamma \sigma$ with $\sigma \in \mathrm{Stab}_{H}(\{y_{1}\})$ such that $\sigma(x_{1})=x_{2}$ and $\gamma \in \mathrm{Stab}_{H}(\{x_{2}\})$ such that $\gamma(y_{1})=y_{2}$ (one can treat  similarly the case $x_1\neq y_2$).
\item[-] If $y_{1}=x_{2}$ and $x_{1}=y_{2}$, we can take $g \in \mathrm{Stab}_G(\{x_1,x_2\}) \setminus \{\mathrm{Id}\} $, which is not empty by assumption on $G$. 
\end{itemize} 
Now, by the main result in \cite{2-transitif}, either $G$ contains $\Agothique_{n}$ or $|G|\leq \exp (\exp (1.18 \sqrt{\log n}))$ for all $n\geq 5\cdot 10^5$, contradicting $|G|\ge 2^{n-1}$ for $n$ big enough.

\section*{Acknowledgments} 
I am very grateful to the anonymous referee for their careful reading of a previous version of this article and their insightful remarks.

\bibliographystyle{alpha}

\end{document}